\documentclass[12pt]{amsart}
\usepackage{amsmath,amsfonts,amsthm,amscd,amssymb,mathrsfs,amssymb}
\usepackage{stmaryrd}
\usepackage{graphicx}
\newtheorem{theorem}{Theorem}

\newtheorem{proposition}[theorem]{Proposition}
\newtheorem{lemma}[theorem]{Lemma}
\newtheorem{definition}[theorem]{Definition}

\newtheorem{corollary}[theorem]{Corollary}

\newtheorem{remark}[theorem]{Remark}

\renewcommand{\H}{\mathcal{H}}
\newcommand{\CP}{\mathbb{CP}}

\newcommand{\RR}{\mathbb{R}}
\newcommand{\ZZ}{\mathbb{Z}}
\newcommand{\RP}{\mathbb{RP}}
\newcommand{\U}{{\rm{U}}}
\newcommand{\LB}{{\rm{LB}}}

\renewcommand{\i}{i}
\newcommand{\w}{\omega}

\numberwithin{equation}{section}
\numberwithin{theorem}{section}
\numberwithin{table}{section}
\numberwithin{table}{section}
\begin{document}
\bibliographystyle{amsalpha} 
\title[Anti-self-dual metrics]{An index theorem for 
anti-self-dual \\ orbifold-cone metrics}
\author{Michael T. Lock}
\address{Department of Mathematics, University of Wisconsin, Madison, 
WI, 53706}
\email{lock@math.wisc.edu}
\author{Jeff A. Viaclovsky}
\address{Department of Mathematics, University of Wisconsin, Madison, 
WI, 53706}
\email{jeffv@math.wisc.edu}
\thanks{Research partially supported by NSF Grant DMS-1105187}
\begin{abstract}
Recently, Atiyah and LeBrun proved versions of the Gauss-Bonnet  
and Hirzebruch signature Theorems for metrics with edge-cone
singularities in dimension four, which they applied to obtain an inequality
of Hitchin-Thorpe type for Einstein edge-cone metrics.
Interestingly, many natural examples of edge-cone metrics in dimension 
four are anti-self-dual (or self-dual depending upon choice of 
orientation). On such a space there is an important elliptic complex called the 
anti-self-dual deformation complex, whose index gives 
crucial information about the local structure of the moduli 
space of anti-self-dual metrics. In this paper, 
we compute the index of this complex in the orbifold case, and give several applications. 
\end{abstract}
\date{September 12, 2012}
\maketitle

\section{Introduction}
\label{I}

We will be concerned with metrics 
with the following type of singularities. 

\begin{definition}
\label{def}
{\em
Let $M$ be a smooth four-manifold with a smoothly embedded 
two-dimensional submanifold $ \Sigma \subset M$. 
We will say that $g$ is an }orbifold-cone metric\em{ on $(M,\Sigma)$ with cone angle $2\pi/p$, where $p\geq 1$ is an integer, if
$g$ is a smooth metric on $M\setminus\Sigma$ and, near any point of $\Sigma$, the metric is locally the quotient of a smooth $\Gamma$-invariant metric on $\mathbb{R}^4$ around the origin for a cyclic group $\Gamma\subset \rm{U}(2)$ with generator given by}
\begin{align*}
(z_1,z_2)\mapsto (z_1,e^{i 2\pi/p}z_2).
\end{align*}
We will refer to $\Sigma$ as the \em{singular set}.
\end{definition}

Around any point $q\in \Sigma$ there exists a neighborhood $U_q=\tilde{U}_q/\Gamma$, where $\tilde{U}_q$ is a neighborhood of the origin in $\mathbb{R}^4$.  We can choose coordinates $(x_1,x_2,y_1, y_2)$ on $\tilde{U}_q$ so that $\Sigma$ is given by $y_1=y_2=0$.  Then, after changing coordinates to $(x_1,x_2,r,\theta)$
by setting $y_1=r \cos(\theta)$ and $y_2=r \sin(\theta)$, the metric on $U_q\cap (M\setminus\Sigma)$ can be expressed as
\begin{align}
g=dr^2+\bigg(\frac{r^2}{p^2}\bigg)d\theta^2+f_{ij}(x)dx^i\otimes dx^j+r^2h,
\end{align}
where the $f_{ij}(x)$ are smooth functions on $\Sigma$ that are symmetric in $i$ and $j$, and $h$ is a smooth symmetric two-tensor field.

Viewing orbifold-cone metrics in this way, it is clear that they form a subclass of edge-cone metrics, which are a generalization of Definition \ref{def} allowing 
arbitrary cone angle $2\pi \beta$ for any $\beta \in \mathbb{R}$, 
see \cite{AtiyahLeBrun} for the full definition.  Edge-cone metrics have recently been of great interest in 
K\"ahler geometry, see for example \cite{Brendle, Donaldson, JMR}. 
For an edge-cone metric $g$, define 
\begin{align}
\label{GB}
\chi_{orb}(M) &= \frac{1}{8\pi^2}\int_{M}\Big(|W|^2-\frac{1}{2}|E|^2+\frac{1}{24}R^2\Big)dV_{g}\\
\label{HST}
\tau_{orb}(M) & = \frac{1}{12\pi^2}\int_{M}(|W^+|^2-|W^-|^2)dV_{g},
\end{align}
where $W$ is the Weyl tensor, $E$ is the traceless Ricci tensor, $R$ is the scalar curvature and 
$W^{\pm}$ are the self-dual and anti-self-dual parts of the Weyl tensor defined below.

We let $[\Sigma]^2$ denote the self-intersection of $\Sigma$ in $M$, which is identified with the Euler class 
of the normal bundle of $\Sigma$ paired with the fundamental class of $\Sigma$.  In the case that the singular set 
is non-orientable, we can understand this by pulling the Euler class of the normal bundle back to the orientable double cover, evaluating it on the corresponding fundamental class and then dividing by two.

In \cite{AtiyahLeBrun}, Atiyah-LeBrun proved the following versions of the 
Gauss-Bonnet and Hirzebruch signature Theorems for edge-cone metrics:
\begin{theorem}[Atiyah-LeBrun \cite{AtiyahLeBrun}]
Let $g$ be an edge-cone metric
on a smooth four-manifold $M$ with singular set
$\Sigma \subset M$ and with cone angle $2 \pi \beta$.
Then 
\begin{align}
\label{euler-signature}
\begin{split}
\chi_{orb}(M)
&=\chi(M) - (1 - \beta) \chi(\Sigma) \\
\tau_{orb}(M)
&=\tau(M) -\frac{1}{3}(1 - \beta^2)[{\Sigma}]^2.
\end{split}
\end{align}
\end{theorem}

As an application, Atiyah-LeBrun proved a version of the 
Hitchin-Thorpe inequality for Einstein edge-cone metrics. 
They also discussed many examples of Einstein edge-cone 
metrics. Interestingly, all of the examples considered in that paper 
also happened to be self-dual or anti-self-dual metrics, 
which we now very briefly describe, and we refer the reader 
to \cite{LockViaclovsky} for more background and details.  

It is well-known that on an oriented four-manifold $M$, the Weyl tensor 
decomposes as $W = W^+ + W^-$, where $W^+$ and $W^-$ are the self-dual 
and anti-self-dual parts of the Weyl tensor, respectively. 
A metric $g$ is said to be anti-self-dual or self-dual if $W^+ \equiv 0$
or $W^- \equiv0$, respectively. In the anti-self-dual case, 
local information of the moduli space of anti-self-dual metrics near $g$ 
is contained in the elliptic complex

\begin{align}
\label{thecomplex}
\Gamma(T^*M) \overset{\mathcal{K}_g}{\longrightarrow} 
\Gamma(S^2_0(T^*M))  \overset{\mathcal{D}^+}{\longrightarrow}
\Gamma(S^2_0(\Lambda^2_+)),
\end{align}
where $\mathcal{K}_g$ is the conformal Killing operator, 
$S^2_0(T^*M)$ denotes traceless symmetric tensors, 
and $\mathcal{\mathcal{D}^+} = (\mathcal{W}^+)_g'$ is the linearized self-dual Weyl curvature 
operator.
In the self-dual case, the relevant complex is 
\begin{align}
\label{thecomplex2}
\Gamma(T^*M) \overset{\mathcal{K}_g}{\longrightarrow} 
\Gamma(S^2_0(T^*M))  \overset{\mathcal{D}^-}{\longrightarrow}
\Gamma(S^2_0(\Lambda^2_-)),
\end{align}
where  $\mathcal{D^-} = (\mathcal{W}^-)_g'$ is the linearized anti-self-dual 
Weyl curvature operator.

If $g$ is a smooth anti-self-dual or self-dual
Riemannian metric, from the Atiyah-Singer Index Theorem, 
the index of the complex \eqref{thecomplex} or \eqref{thecomplex2} 
is given by 
\begin{align}
\label{manifoldindex}
Ind(M,g) 
= \dim( H^0) -  \dim( H^1) + \dim( H^2) = \frac{1}{2} ( 15 \chi(M) \pm 29 \tau(M)), 
\end{align}
where $\chi(M)$ is the Euler characteristic,
$\tau(M)$ is the signature of $M$, and
 $H^i$ is the $i$th cohomology of the complex \eqref{thecomplex}
in the positive case, and the complex \eqref{thecomplex2} in the 
negative case,  for $i = 0,1,2$; see \cite{KotschickKing}.

For an orbifold-cone metric, the index is computed by looking
at smooth sections in the orbifold sense, see Section \ref{tgi}.  
In this setting, the formula given in  \eqref{manifoldindex}
is not necessarily correct, and there are correction terms required
arising from the singularities.

We note that the complex \eqref{thecomplex} yields local information about the
structure of the moduli space of anti-self-dual metrics near $g$. 
That is, there is a map, called the {\em{Kuranishi map}}
\begin{align}
\Psi : H^1 \rightarrow H^2
\end{align}
which is equivariant under the action of $H^0$, such that the 
moduli space of anti-self-dual orbifold-cone metrics with singular set 
$\Sigma \subset M$ and fixed cone angle $2 \pi /p$ near $g$, 
$\mathcal{M}_g$, is locally isomorphic to 
$\Psi^{-1}(0) / H^0$. This is a standard fact in the setting of 
smooth manifolds, and the proof of existence of the 
Kuranishi map readily  generalizes to the setting of orbifold-cone metrics. 
It is important to note that this map does not take into account deformations 
of the cone angle. 

In a previous paper, the authors proved an extension of the index formula 
to anti-self-dual orbifold metrics with isolated cyclic quotient 
singularities \cite{LockViaclovsky}. In this paper, we prove an 
extension of the index formula \eqref{manifoldindex} 
to anti-self-dual metrics with orbifold-cone singularities:
\begin{theorem}
\label{mainthm}
Let $g$ be an orbifold-cone metric
on a smooth four-manifold $M$ with singular set
$\Sigma \subset M$ and cone angle $2 \pi/p$. 
If $g$ is anti-self-dual, then 
the index of the complex \eqref{thecomplex} is given 
by 
\begin{align}
\label{Ind}
Ind^{ASD}(M, g) =\frac{1}{2}(15\chi(M)+29\tau(M))-4\chi(\Sigma)-4[\Sigma]^2,
\end{align}
where $[\Sigma]^2$ denotes the self-intersection number of $\Sigma$ in $M$. 
If $g$ is instead self-dual, then the index of the complex \eqref{thecomplex2} is given by
\begin{align}
\label{Ind2}
Ind^{SD}(M, g) =\frac{1}{2}(15\chi(M)-29\tau(M))-4\chi(\Sigma)+4[\Sigma]^2.
\end{align}
\end{theorem}

We emphasize that the index is independent of the cone angle and only depends on the 
topologies of $M$ and $\Sigma$, and the embedding of $\Sigma$ into $M$.
Our proof of this is an application of Kawasaki's orbifold index 
theorem from \cite{Kawasaki}.

\begin{remark}
{\em There are many examples of continuous families of anti-self-dual edge-cone metrics arising from deforming the cone angle, 
see \cite{AtiyahLeBrun}.  In all such examples the above index formula can be seen to hold for arbitrary real cone angles.  
Thus, it is likely that this formula holds in general for anti-self-dual edge-cone metrics with arbitrary cone angle $2\pi \beta$.  However, in this case the index must be defined using appropriate weighted edge H\"older spaces to obtain Fredholm 
operators.  This introduces considerable technical complications, and we plan to address this in a forthcoming paper.}
\end{remark}

It is useful to make the following definition.
\begin{definition}
{\em An anti-self-dual (self-dual) orbifold-cone metric with $H^2=\{0\}$ is called} unobstructed.
\end{definition}

It was conjectured by I.M. Singer in 1978 that a positive
scalar curvature anti-self-dual metric is unobstructed.
The evidence for this conjecture is very strong, but it has
not yet been proven in full generality. However, the conjecture
is certainly true in the Einstein case:

\begin{lemma}
\label{lemma}
Any anti-self-dual (self-dual) Einstein orbifold-cone metric with positive scalar curvature is unobstructed.
\end{lemma}
This was proved in the smooth case by \cite{Itoh}, and in Section \ref{coro}
we will show that his proof extends to the orbifold-cone setting.

\subsection{Self-dual edge-cone metrics on $S^4$}
The first examples we consider were found by Hitchin in \cite{Hitchin1993}.
They are a family of self-dual Einstein orbifold-cone metrics on $S^4$ with singular set 
an $\RP^2$ and cone angle $2 \pi / (k -2)$, where $k \geq 3$ is an integer.  These metrics have the $3$-dimesional 
isometry group $\rm{SO}(3)$.  These metrics are self-dual, which determines an orientation on $S^4$.  Everything 
we say below is with respect to this orientation.

The singular set is a Veronese $\RP^2\subset S^4$.  This arises by first looking at the representation of $S^2_0(\mathbb{R}^3)$ 
in $\mathbb{R}^5$, which yields an embedding 
$\mathbb{RP}^2\hookrightarrow \mathbb{RP}^4$, having two lifts into $S^4$ of self-intersection
$\pm2$  respectively.  The singular set for Hitchin's metrics is 
the $\RP^2$ with self-intersection $-2$.   We have the following rigidity result for Hitchin's metrics:
\begin{corollary}
\label{c1}
For any $k \geq 3$, a Hitchin metric on $(S^4, \RP^2)$ is rigid as a 
self-dual orbifold-cone metric with cone angle $2 \pi / (k -2)$.
\end{corollary}

Hitchin's metrics all have singular set an embedded $\RP^2$. A
natural question is whether there are self-dual Einstein orbifold-cone
metrics on $S^4$ with other singular sets.
Atiyah-LeBrun give the following family of examples on $(S^4,S^2)$ \cite[page~21]{AtiyahLeBrun}: 
$S^4 \setminus S^2$ is 
conformally isometric  to $\mathcal{H}^3 \times S^1$, with the 
product metric $h + d \theta^2$ where $\mathcal{H}^3$ is the $3$-dimensional 
hyperbolic upper half plane with hyperbolic metric $h$. 
The metric $h + \beta^2 d\theta^2$ is then a constant curvature
edge-cone metric with angle $2 \pi \beta$, with singular set $S^2$. 

\begin{corollary}
\label{orientable}
No unobstructed self-dual orbifold-cone metrics exist on $S^4$ with singular set diffeomorphic to a genus $j\geq1$ 
orientable surface.
\end{corollary}

The natural question is then whether there are unobstructed self-dual 
orbifold-cone metrics on $S^4$ with singular set a smoothly embedded surface 
diffeomorphic to $j\#\RP^2$ when $j>1$, which we denote here by $\Sigma^j$.

For $\Sigma^j$ embedded in $S^4$,
Whitney studied the possible values of the self-intersection number  
and proved that $[\Sigma^j]^2\equiv 2\chi(\Sigma^j)\text{ mod } 4$. 
He also proposed a conjecture, which Massey later proved \cite{Massey}, 
stating that $[\Sigma^j]^2$ could only take the following values:
\begin{align}
\label{values}
2\chi(\Sigma^j)-4,\phantom{=} 2\chi(\Sigma^j),\phantom{=} 2\chi(\Sigma^j)+4, \cdots, \phantom{=}4-2\chi(\Sigma^j).
\end{align}
Since $\chi(\Sigma^j)=2-j$, this set of values can be written in terms of $j$ as:
\begin{align}
-2j,\phantom{ }-2j+4, \phantom{ }-2j+8, \cdots, \phantom{ }2j.
\end{align}
Moreover, Massey also proved that any of these values can be obtained
 by an appropriate embedding of $\Sigma^j$ in $S^4$.  The next result gives a restriction on the self-intersection number of the singular set for such a metric:
\begin{corollary}
\label{c2}
If $g$ is an unobstructed self-dual orbifold-cone 
metric on $(S^4, \Sigma^j)$, where $\Sigma^j$ is diffeomorphic to $j\#\RP^2$ when $j\geq 1$,
then we have the inequalities
\begin{align}
-2j\leq[\Sigma^j]^2<-j.
\end{align}
\end{corollary}
It is an interesting problem to find examples of unobstructed self-dual orbifold-cone metrics 
on $(S^4,\Sigma^j)$ when $j > 1$.

\subsection{LeBrun's hyperbolic monopole metrics}

Next, we turn to LeBrun's hyperbolic monopole metrics from \cite{LeBrunJDG}.
These metrics are defined similarly to the Gibbon-Hawking 
multi-Eguchi-Hanson metrics \cite{GibbonsHawking}, 
but with hyperbolic $3$-space $\mathcal{H}^3$ replacing Euclidean $3$-space. 
To define these metrics, first choose $n$ points $\{p_i\}$ in 
hyperbolic $3$-space, and let
\begin{align}
V = 1 + \sum_{i=1}^n \Gamma_{p_i}
\end{align}
where $\Gamma_{p_i}$ is the hyperbolic Green's function based at $p_i$ 
with normalization $\Delta \Gamma_{p_i} =  -2 \pi \delta_{p_j}$. 
Letting $P$ denote the collection of monopole points $p_i$, 
$* dV$ is a closed $2$-form on $\H^3 \setminus P$,
and $(1/ 2 \pi)[* dV]$ is an integral class in $
H^2 ( \H^3 \setminus P, \ZZ )$.
Let $\pi: X_0 \rightarrow \H^3 \setminus P$ 
be the unique principal $\U(1)$-bundle determined by the 
the above integral class.
By Chern-Weil theory, there is a connection form $\w \in H^1(X_0, \i \RR)$
with curvature form $\i (* dV)$. LeBrun's metric is defined by 
\begin{align}
\label{LBmetric}
g_{\LB} =   V \cdot g_{\H^3} - V^{-1} \w \odot \w.
\end{align}
Next define a larger manifold $X$ by attaching points $\tilde{p}_j$ 
over each $p_j$, and by adding $\Sigma = S^2$ corresponding to
the boundary of hyperbolic space. By an appropriate 
choice of conformal factor, the metric extends smoothly 
to this compactification, which is diffeomorphic to $n \# \CP^2$. 
All of these conformal classes admit an $S^1$-action. 

In \cite[Section 5]{AtiyahLeBrun} it was noted that by 
replacing $V$ with 
\begin{align}
V = \beta^{-1} + \sum_{i=1}^n \Gamma_{p_i},
\end{align}
one obtains an self-dual edge-cone metric on $(n \# \CP^2, \Sigma)$,
where $\Sigma = S^2$, with cone angle $2 \pi \beta$. 

For $\beta = 1$ and $n \geq 3$, it is well-known that LeBrun's metrics 
admit non-$S^1$-equivariant deformations,
with the moduli space locally of dimension $7n - 15$ \cite{LeBrunJAMS}.
However, for orbifold-cone metrics, assuming these metrics are unobstructed, then somewhat surprisingly 
this is no longer true: 
\begin{corollary}
\label{c4}
Let $g$ be an unobstructed LeBrun self-dual orbifold-cone metric on $(n \# \CP^2, \Sigma)$ 
with cone angle $2 \pi /p$. 
Then any self-dual edge-cone deformation of $g$ with cone angle 
$2 \pi /p$ also admits an $S^1$-action. Thus the moduli 
space of such metrics near $g$ is of dimension $3 (n -2)$
for $n \geq 3$. 
\end{corollary}
Unobstructedness is true in the smooth case ($\beta =1$),
but the proof of this relies on tools which do not easily
generalize to the orbifold-cone case. However, we do expect that
these metrics are also unobstructed for $\beta \neq 1$,
so this assumption is probably not necessary.

\subsection{Ricci-flat anti-self-dual metrics}

Finally, we consider the Ricci-flat case. There are many known example 
of such metrics with edge-cone singularities. For example, 
in \cite{Brendle} examples of K\"ahler Ricci-flat 
metrics with edge-cone singularities were obtained. 
These are of relevance to this paper, since such a metric in 
dimension $4$ is in necessarily anti-self-dual.  
The following result computes the dimension of the moduli space of 
anti-self-dual metrics in the more general Ricci-flat anti-self-dual case:
\begin{corollary}
\label{c5}
Let $g$ be an anti-self-dual Ricci-flat orbifold-cone metric on $M$ 
with singular set $\Sigma$ and cone angle $2 \pi /p$. 
Assume that there are no parallel 
vector fields, and also that there are no parallel 
sections of $S^2_0( \Lambda^2_+)$ with respect to $g$. 
Then the moduli space of anti-self-dual orbifold-cone metrics 
on $M$ with singular set $\Sigma$ and cone angle $2 \pi /p$
near $g$ is a smooth manifold $\mathcal{M}_g$ of dimension
\begin{align}
\label{rfdim}
\dim( \mathcal{M}_g) = - \frac{1}{2}(15\chi(M)+29\tau(M)) + 4\chi(\Sigma) +4[\Sigma]^2.
\end{align}
\end{corollary}
Since there
are many examples of Kahler-Ricci flat edge-cone
metrics which are not orbifold-cone metrics, it is a very interesting
problem to generalize Corollary 1.10 to the more general edge-cone case.
We expect the same formula holds for any cone angle.

We end with a brief outline of the paper. In Section \ref{setup}, 
we give the necessary background and set-up. In Section \ref{ecc}, 
we compute the required equivariant Chern characters. 
Theorem \ref{mainthm} is then proved in Section \ref{tgi}. 
Finally, in Section \ref{coro}, the proofs of 
Corollaries \ref{c1}--\ref{c5} are given. 

\subsection{Acknowledgements}
The authors would like to thank Claude LeBrun for helpful remarks.

\section{Analysis of $T^*M_{|_{\Sigma}}$ and $\Lambda^2_{\pm|_{\Sigma}}$}
\label{setup}
To compute the index it is necessary to understand how the pullback of the complexified principle symbol of the complex $\eqref{thecomplex}$, and the $K$-theoretic Thom class of the complexified normal bundle decompose into line bundles, and how the orbifold structure group acts on these decompositions.  This will be dealt with in Section $\ref{ecc}$.  In this section, however, we will analyze $T^*M_{|_{\Sigma}}$ and $\Lambda^2_{\pm|_{\Sigma}}$ because these bundles play a crucial role in the decompositions in Section $\ref{ecc}$.  For the rest of the paper we will denote the complexification of a real bundle, $E$, by $E_{\mathbb{C}}$.  

We begin by decomposing, in real coordinates,
\begin{align}
T^*M_{|_\Sigma}=T^*\oplus N^*,
\end{align}
which are the tangent and normal bundles to the singular set respectively.  In local orthonormal coordinates we can write
\begin{align}
\begin{split}
T^*&=\text{span}\{e^1,e^2\} \text{ and,}\\
N^*&=\text{span}\{e^3,e^4\}.
\end{split}
\end{align}
Using these coordinates we describe the orbifold structure group.  We are considering orbifold-cone metrics with
cone angle $2\pi/p$ so the orbifold structure group is the cyclic group $\Gamma$, of order $p$, consisting of elements $\gamma_j$, which can be locally written as
\begin{align}
\gamma_j=\left(
\begin{matrix}
1&0&0&0\\
0&1&0&0\\
0&0&\cos(\frac{2\pi}{p}j)& -\sin(\frac{2\pi}{p}j)\\
0&0&\sin(\frac{2\pi}{p}j)&\cos(\frac{2\pi}{p}j)\\
\end{matrix}
\right).
\end{align}
We will often refer to a general element $\gamma \in \Gamma$ and denote the angle of the corresponding action by $\theta$.  

We have the following sections of $T^*_{\mathbb{C}}$:
\begin{align}
\begin{split}
&\alpha_1=e^1+ie^2 \text{ and}\\
&\bar{\alpha}_1=e^1-ie^2,
\end{split}
\end{align}
and the following sections of $N^*_{\mathbb{C}}$:
\begin{align}
\begin{split}
&\alpha_2=e^3+ie^4 \text{ and}\\
&\bar{\alpha}_2=e^3-ie^4.
\end{split}
\end{align}
Now consider the line bundles over $F$:
\begin{align}
\begin{split}
\Theta_i&=\text{span}\{\alpha_i\} \text{ and}\\
\bar{\Theta}_i&=\text{span}\{\bar{\alpha}_i\}.  
\end{split}
\end{align}
It is clear that we have the line bundle decompositions:
\begin{align}
\begin{split}
\label{dec1}
T^*_{\mathbb{C}}&=\Theta_1\oplus \bar{\Theta}_1 \text{ and}\\
N^*_{\mathbb{C}}&=\Theta_2\oplus \bar{\Theta}_2,
\end{split}
\end{align}
and in these coordinates for $N^*_{ \mathbb{C}}$ we have
\begin{align}
\label{gamma-N}
\gamma_{|_{N^*_{\mathbb{C}}}}=\left(
\begin{matrix}
e^{i\theta} & 0 \\
0& e^{-i\theta} \\
\end{matrix}
\right).
\end{align}

Recall that $\Lambda^2_{+\mathbb{C}|_\Sigma}$ is generated by three sections, which can be locally written as
\begin{align}
\begin{split}
&\omega^+_{1_{\mathbb{C}}}=e^1\wedge e^2+ e^3\wedge e^4\\
&\omega^+_{2_{\mathbb{C}}}=e^1\wedge e^3+ e^4\wedge e^2=\frac{1}{2}(\alpha_1\wedge \alpha_2 +\bar{\alpha}_1\wedge \bar{\alpha}_2)\\
&\omega^+_{3_{\mathbb{C}}}=e^1\wedge e^4+ e^2\wedge e^3=-\frac{i}{2}(\alpha_1\wedge \alpha_2 -\bar{\alpha}_1\wedge \bar{\alpha}_2),
\end{split}
\end{align}
Since $\omega^+_{1_{\mathbb{C}}}$ is a global non-zero section, it spans a trivial line bundle which we will denote by $\mathbb{C}_+$.
Therefore, we can decompose $\Lambda^2_{+\mathbb{C}|_\Sigma}$ into line bundles as
\begin{align}
\label{dec+}
\Lambda^2_{+\mathbb{C}|_\Sigma}=\mathbb{C}_+\oplus (\Theta_1\otimes\Theta_2) \oplus (\bar{\Theta}_1\otimes \bar{\Theta}_2),
\end{align}
which, with respect to this decomposition, admits the group action
\begin{align}
\gamma_{|\Lambda^2_{+\mathbb{C}|_\Sigma}}=\left(
\begin{matrix}
1 & 0 & 0\\
0& e^{i\theta} & 0 \\
0&0& e^{-i\theta} \\
\end{matrix}
\right).
\end{align}

Similarly, recall that $\Lambda^2_{-\mathbb{C}|_\Sigma}$ is generated by three sections, which can be locally written as
\begin{align}
\begin{split}
&\omega^-_{1_{\mathbb{C}}}=e^1\wedge e^2- e^3\wedge e^4\\
&\omega^-_{2_{\mathbb{C}}}=e^1\wedge e^3- e^4\wedge e^2=\frac{1}{2}(\alpha_1\wedge \bar{\alpha}_2 +\bar{\alpha}_1\wedge \alpha_2)\\
&\omega^-_{3_{\mathbb{C}}}=e^1\wedge e^4- e^2\wedge e^3=-\frac{i}{2}(\alpha_1\wedge \bar{\alpha}_2 +\bar{\alpha}_1\wedge \alpha_2),
\end{split}
\end{align}
Here we will denote the trivial line bundle that is the span of $\omega^-_{1_{\mathbb{C}}}$ as $\mathbb{C}_-$.
Therefore, we can decompose $\Lambda^2_{-\mathbb{C}|_\Sigma}$ into line bundles as
\begin{align}
\label{dec-}
\Lambda^2_{-\mathbb{C}|_\Sigma}=\mathbb{C}_-\oplus  (\bar{\Theta}_1\otimes \Theta_2)\oplus (\Theta_1\otimes\bar{\Theta}_2),
\end{align}
which, with respect to this decomposition, admits the group action
\begin{align}
\gamma_{|\Lambda^2_{-\mathbb{C}|_\Sigma}}=\left(
\begin{matrix}
1 & 0 & 0\\
0& e^{i\theta} & 0 \\
0&0& e^{-i\theta} \\
\end{matrix}
\right).
\end{align}

\section{Equivariant Chern characters}
\label{ecc}
Throughout the rest  of this paper, we will denote the Euler class of $T^*$ by $e$, and the orbifold Euler class of $N^*$ by $\hat{h}$.  We will also denote their pairings with the fundamental class of $\Sigma$, $\langle e,[\Sigma]\rangle$ and $\langle\hat{h},\Sigma\rangle$, by $\chi(\Sigma)$ and $[\hat{\Sigma}]^2$ respectively.  It is important to notice that if we consider $\Sigma$ as a smoothly embedded submanifold with the regular Euler class of its normal bundle $h$, then $[\Sigma]^2=\langle h,[\Sigma]\rangle=p[\hat{\Sigma}]^2$, the self-intersection number of $\Sigma$ in $M$, where $p$ comes from the cone angle.  Also, for the remainder of the construction of the index we will assume that $\Sigma$ is orientable.  This is a necessary assumption for the Index theorem.  However, once we prove Theorem $\ref{mainthm}$ for $\Sigma$ orientable, it is very easy to show that it also holds for $\Sigma$ non-orientable.

We will frequently make use of the equivariant Chern characters of the complex line bundles in decomposition $\eqref{dec1}$.  Since $\gamma$ acts trvially on $\Theta_1\oplus \bar{\Theta}_1$, and acts on $\Theta_2\oplus \bar{\Theta}_2$ as in $\eqref{gamma-N}$, we see that
\begin{align}
\begin{split}
&ch_{\gamma}(\Theta_1)=ch(\Theta_1)=\sum_{j=0}^{\infty}\frac{e^j}{j!}\\
&ch_{\gamma}(\bar{\Theta}_1)=ch(\bar{\Theta}_1)=\sum_{j=0}^{\infty}\frac{(-e)^j}{j!}\\
&ch_{\gamma}(\Theta_2)=e^{i\theta}ch(\Theta_2)=e^{i\theta}\sum_{j=0}^{\infty}\frac{\hat{h}^j}{j!}\\
&ch_{\gamma}(\bar{\Theta}_2)=e^{-i\theta}ch(\bar{\Theta}_1)=e^{-i\theta}\sum_{j=0}^{\infty}\frac{(-\hat{h})^j}{j!}.
\end{split}
\end{align}

To find the anti-self-dual index, we need to compute the equivariant Chern character on the pullback of the complexified principle symbol over $\Sigma$, $i^*\sigma$:
\begin{align}
i^*\sigma=i^*[T^*M_{\mathbb{C}}]-i^*[S^2_0T^*M_{\mathbb{C}}]+i^*[S^2_0\Lambda^2_{+\mathbb{C}|_\Sigma}],
\end{align}
where $i:\Sigma \rightarrow M$ is the inclusion of the singular set $\Sigma$ into the orbifold $M$.  We will also need to compute the equivariant Chern character of the $K$-theoretic Thom class of the complexified normal bundle:
\begin{align}
\lambda_{-1}N^*_{\mathbb{C}}=[\Lambda^0N^*_{\mathbb{C}}]-[\Lambda^1N^*_{\mathbb{C}}]+[\Lambda^2N^*_{\mathbb{C}}]
\end{align}

We begin this section by computing $ch_{\gamma}(\Lambda^2_{\pm\mathbb{C}|_\Sigma})$, next we compute $ch_{\gamma}(i^*\sigma)$ and finally we compute $ch_{\gamma}(\lambda_{-1}N^*_{\mathbb{C}})$.

\subsection{Equivariant Chern characters of $\Lambda^2_{\pm\mathbb{C}|_\Sigma}$}
Using the decomposition $\eqref{dec+}$ of $\Lambda^2_{+\mathbb{C}|_\Sigma}$,  we have that
\begin{align}
ch_{\gamma}(\Lambda^2_{+\mathbb{C}|_\Sigma})=ch_{\gamma}(\mathbb{C})+ch_{\gamma} (\Theta_1\otimes\Theta_2) +ch_{\gamma} (\bar{\Theta}_1\otimes \bar{\Theta}_2).
\end{align}
The first term on the right hand side is $1$ because the $\gamma$-action on $\mathbb{C}_+$ is trivial.  We compute the second two terms on the right hand side as:
\begin{align*}
ch_{\gamma}(\Theta_1\otimes \Theta_2)&=ch_{\gamma}(\Theta_1)\cdot ch_{\gamma}(\Theta_2)=(1+e+\frac{e^2}{2}+\cdots)\cdot e^{i\theta}(1+\hat{h}+\frac{h^2}{2}+\cdots)\\
&=e^{i\theta}(1+e+\hat{h}+e\hat{h}+\frac{e^2}{2}+\frac{\hat{h}^2}{2}+\cdots) \text{, and }\\
ch_{\gamma}(\bar{\Theta}_1\otimes \bar{\Theta}_2)&=ch_{\gamma}(\bar{\Theta}_1)\cdot ch_{\gamma}(\bar{\Theta}_2)=(1-e+\frac{e^2}{2}+\cdots)\cdot e^{-i\theta}(1-h+\frac{\hat{h}^2}{2}+\cdots)\\
&=e^{-i\theta}(1-e-\hat{h}+e\hat{h}+\frac{e^2}{2}+\frac{\hat{h}^2}{2}+\cdots).
\end{align*}
Therefore, we can combine these terms to find
\begin{align}
\label{Lambda+}
ch_{\gamma}(\Lambda^2_{+\mathbb{C}|_\Sigma})=1+\cos(\theta)(2+2e\hat{h}+e^2+\hat{h}^2+\cdots)+i\sin(\theta)(2e+2\hat{h}+\cdots).
\end{align}
Similarly, we find that
\begin{align}
\label{Lambda-}
ch_{\gamma}(\Lambda^2_{-\mathbb{C}|_\Sigma})=1+\cos(\theta)(2-2e\hat{h}+e^2+\hat{h}^2+\cdots)+i\sin(\theta)(-2e+2\hat{h}+\cdots).
\end{align}

\subsection{Equivariant Chern character of $i^*\sigma$}
We will begin by computing the equivariant Chern characters of the individual $K$-theoretic classes that compose $i^*\sigma$ and then sum them accordingly to find $ch_{\gamma}(i^*\sigma)$.  

First consider the bundle $i^*(T^*M_{\mathbb{C}})=T^*_{\mathbb{C}}\oplus N^*_{\mathbb{C}}$.  We have
\begin{align}
\begin{split}
ch_{\gamma}&(i^*[T^*M_{\mathbb{C}}])=ch_{\gamma}(T^*_{\mathbb{C}}\oplus N^*_{\mathbb{C}})\\
&=(2+e^2+\cdots)+e^{i\theta}(1+\hat{h}+\frac{\hat{h}^2}{2}+\cdots)+e^{-i\theta}(1-\hat{h}+\frac{\hat{h}^2}{2}+\cdots)\\
&=(2+e^2+\cdots)+\cos(\theta)(2+\hat{h}^2+\cdots)+i\sin(\theta)(2\hat{h}+\cdots).
\end{split}
\end{align}

Next, consider the bundle $i^*(S^2_0T^*M_{\mathbb{C}})$.  Using the formulas $\eqref{Lambda+}$ and $\eqref{Lambda-}$, and the bundle isomorphism $S^2_0T^*M=\Lambda^2_+\otimes \Lambda^2_-$ we compute
\begin{align}
\begin{split}
ch_{\gamma}&(i^*[S^2_0T^*M_{\mathbb{C}}])=ch_{\gamma}(i^*\Lambda^2_{+\mathbb{C}}\otimes i^*\Lambda^2_{-\mathbb{C}})=ch_{\gamma}(\Lambda^2_{+\mathbb{C}|_\Sigma})\cdot ch_{\gamma}\Lambda^2_{-\mathbb{C}|_\Sigma})\\
&=\big[1+4\cos(\theta)+4\cos^2(\theta)\big]+\hat{h}\big[i4\sin(\theta)+i8\sin(\theta)\cos(\theta)\big]\\
&\phantom{==}+e^2\big[4+2\cos(\theta)\big]+\hat{h}^2\big[-4+2\cos(\theta)+8\cos^2(\theta)\big]+\cdots.
\end{split}
\end{align}

Finally, consider the bundle $i^*(S^2_0\Lambda^2_{+\mathbb{C}})=S^2_0\Lambda^2_{+\mathbb{C}|_\Sigma}$, which decomposes as
\begin{align}
\label{decom}
S^2_0\Lambda^2_{+\mathbb{C}|_\Sigma}=\{( \Theta_1\otimes \Theta_2)\oplus (\bar{\Theta}_1\otimes \bar{\Theta}_2)\}
\oplus S^2_0\big( (\Theta_1\otimes \Theta_2)\oplus (\bar{\Theta}_1\otimes \bar{\Theta}_2)\big)\oplus \mathbb{C}_{tr}
\end{align}
where $\mathbb{C}_{tr}$ is a trivial line bundle, with trivial $\gamma$-action, corresponding to the trace term.  
Using this, we are able to compute
\begin{align*}
ch_{\gamma}(i^*[S^2_0&\Lambda^2_+])=ch_{\gamma}\big(\Theta_1\otimes \Theta_2\oplus \bar{\Theta}_1\otimes \bar{\Theta}_2\big) 
+ ch_{\gamma}\big(S^2_0\big(\Theta_1\otimes \Theta_2\oplus \bar{\Theta}_1\otimes \bar{\Theta}_2\big)\big)+1\\
&=\big[\cos(\theta)(2+2eh+e^2+\hat{h}^2+\cdots)+i\sin(\theta)(2e+2\hat{h}+\cdots)\big]\\
&\phantom{=}+\big[\big(\cos(\theta)(2+2e\hat{h}+e^2+\hat{h}^2+\cdots)+i\sin(\theta)(2e+2\hat{h}+\cdots)\big)^2-2\big]+1\\
&=\big[-1+2\cos(\theta)+4\cos^2(\theta)\big]+e\big[i2\sin(\theta)+i8\sin(\theta)\cos(\theta)\big]\\
&\phantom{=}+\hat{h}\big[i2\sin(\theta)+i8\sin(\theta)\cos(\theta)\big]+eh\big[-8+2\cos(\theta)+16\cos^2(\theta)\big]\\
&\phantom{=}+e^2\big[-4+\cos(\theta)+8\cos^2(\theta)\big]+\hat{h}^2\big[-4+\cos(\theta)+8\cos^2(\theta)\big]+\cdots.
\end{align*}

Now, we are able to compute the $ch_{\gamma}(i^*\sigma)$ by taking the appropriate sum of the above Chern characters:
\begin{align}
\begin{split}
ch_{\gamma}(i^*\sigma)&=ch_{\gamma}(i^*[T^*M_{\mathbb{C}}])-ch_{\gamma}(i^*[S^2_0T^*M_{\mathbb{C}}])+ch_{\gamma}(i^*[S^2_0\Lambda^2_+])\\
&=e\big[2i\sin(\theta)+8i\sin(\theta)\cos(\theta)\big]+e\hat{h}\big[-8+2\cos(\theta)+16\cos^2(\theta)\big]\\
&\phantom{==}+e^2\big[8\cos^2(\theta)-\cos(\theta)-7\big]+\cdots.
\end{split}
\end{align}

\subsection{Equivariant Chern character of $\lambda_{-1}N^*_{\mathbb{C}}$}
We begin by examining the bundles representing the $K$-theoretic classes that compose $\lambda_{-1}N^*_{\mathbb{C}}$.  

First, both $\Lambda^0 N^*_{\mathbb{C}}$ and $\Lambda^2 N^*_{\mathbb{C}}$ have non-vanishing global sections, so they are trivial, and clearly admit a trivial $\gamma$ action.  Next, it is clear that $\Lambda^1 N^*_{\mathbb{C}}=N^*_{\mathbb{C}}$.  Therefore, we find that
\begin{align}
\begin{split}
\label{den}
ch_{\gamma}(\lambda_{-1}N^*_{\mathbb{C}})&=ch_{\gamma}(\Lambda^0N^*_{\mathbb{C}})-ch_{\gamma}(\Lambda^1N^*_{\mathbb{C}})+ch_{\gamma}(\Lambda^2N^*_{\mathbb{C}})\\
&=2-ch_{\gamma}(N^*_{\mathbb{C}})\\
&=2-\cos(\theta)(2+\hat{h}^2)-i\sin(\theta)(2\hat{h})+\cdots.
\end{split}
\end{align}

\section{The index}
\label{tgi}
We begin this section with some remarks on the definition of the index in the orbifold case.  As mentioned in the introduction, the index is computed by looking at smooth sections in the orbifold sense.
To define this, we recall that
an orbifold vector bundle is defined in terms of orbifold charts.  Over a neighborhood $U_x$ away from 
$\Sigma$ it is defined as a vector bundle in the usual sense, and over a neighborhood 
$U_q=\tilde{U}_q/\Gamma$ around $q\in \Sigma$, where 
$\tilde{U}_q$ is a neighborhood of the origin in $\mathbb{R}^4$, 
it is identified with the quotient of a smooth $\Gamma$-equivariant vector bundle 
over $\tilde{U}_q$.  On overlaps the obvious compatibility conditions are satisfied.  Smooth sections of an orbifold vector bundle are globally defined sections on $M$.  On a neighborhood $U_x$ away from 
$\Sigma$ it is smooth in the ordinary sense, and on a neighborhood $U_q$ of $q\in \Sigma$, it is identified with a smooth $\Gamma$-equivariant section of the corresponding $\Gamma$-equivariant bundle over $\tilde{U}_q$ defining the orbifold vector bundle in that neighborhood.

With this understandng of the index, from \cite{Kawasaki} and \cite{LM}, recall that the anti-self dual index for $(M,g)$, where $g$ is 
and orbifold-cone metric with singular set $\Sigma$, is given by
\begin{align*}
Ind^{ASD}(M,g)=\frac{1}{2}(15\chi_{orb}(M)+29\tau_{orb}(M))-\bigg\langle\frac{1}{|\Gamma|}\sum_{\gamma \neq Id}\frac{ch_{\gamma}(i^*\sigma)}{ch_{\gamma}(\lambda_{-1}N_{\mathbb{C}})e}\hat{A}(\Sigma)^2,[\Sigma]\bigg\rangle.
\end{align*}
Note that Kawaski's formula is written in terms of evaluation on the 
orbifold tangent bundle of the singular set, but writing it in terms of evaluation on the
fundamental class of $\Sigma$ introduces the Euler class in the 
denominator.

Next, using the formulas \eqref{GB} and \eqref{HST}, 
we can rewrite
\begin{align*}
Ind^{ASD}(M,g)=\frac{1}{2}(15\chi_{top}(M)+29\tau_{top}(M))&-\frac{15}{2}\bigg(\frac{p-1}{p}\bigg)\chi(\Sigma)-\frac{29}{6}\bigg(\frac{p^2-1}{p}\bigg)[\hat{\Sigma}]^2\\
&-\bigg\langle\frac{1}{p}\sum_{j=1}^{p-1}\frac{ch_{\gamma_j}(i^*\sigma)}{ch_{\gamma_j}(\lambda_{-1}N_{\mathbb{C}})e}\hat{A}(\Sigma)^2,[\Sigma]\bigg\rangle.
\end{align*}

\subsection{Computation of correction terms}  
Using the computation of the denominator, $\eqref{den}$, we have
\begin{align*}
\big[ch_{\gamma}(\lambda_{-1}N_{\mathbb{C}})\big]^{-1}&=\big[{2-\cos(\theta)(2+\hat{h}^2)-i\sin(\theta)(2\hat{h})}\big]^{-1}\\
&=\Big[(2-2\cos(\theta))[1-\frac{1}{2-2\cos(\theta)}(\cos(\theta)(\hat{h}^2)+i\sin(\theta)(2\hat{h}))]\Big]^{-1}\\
&=\Big[(2-2\cos(\theta))[1-\mathbb{D}]\Big]^{-1}
\end{align*}
Then, by using a geometric series, we see that
\begin{align}
\big[ch_{\gamma}(\lambda_{-1}N_{\mathbb{C}})\big]^{-1}=\frac{1}{(2-2\cos(\theta))}[1+\mathbb{D}+\mathbb{D}^2+\cdots].
\end{align}
Now, we compute:
\begin{align*}
&\frac{ch_{\gamma}(i^*\sigma)}{ch_{\gamma}(\lambda_{-1}N_{\mathbb{C}})e}=\frac{ch_{\gamma}(i^*\sigma)}{(2-2\cos(\theta))e}[1+\mathbb{D}+\mathbb{D}^2+\cdots]\\
&\phantom{===}=\frac{1}{2}e\frac{(8\cos(\theta)+7)(\cos(\theta)-1)}{(1-\cos(\theta))}+\hat{h}\bigg[-\frac{4}{1-\cos(\theta)}+\frac{\cos(\theta)}{1-\cos(\theta)}+\frac{8\cos^2(\theta)}{1-\cos(\theta)}\bigg]\\
&\phantom{=====}+\bigg[\frac{2i\sin(\theta)+8i\sin(\theta)\cos(\theta)}{2-2\cos(\theta)}\bigg]\cdot\bigg[\hat{h}\frac{2i\sin(\theta)}{2-2\cos(\theta)}\bigg]+\cdots\\
&\phantom{===}=-\frac{1}{2}e[8\cos(\theta)+7]+\hat{h}\bigg[-9-8\cos(\theta)+\frac{5}{1-\cos(\theta)}\bigg]\\
&\phantom{=====}+\hat{h}\bigg[-\frac{1+5\cos(\theta)+4\cos^2(\theta)}{1-\cos(\theta)}\bigg]+\cdots\\
&\phantom{===}=-\frac{1}{2}e[8\cos(\theta)+7]+\hat{h}\bigg[-4\cos(\theta)-\frac{5}{1-\cos(\theta)}\bigg]+\cdots.
\end{align*}
Finally, we find:
\begin{align}
\label{sum}
\frac{1}{p}\sum_{j=1}^{p-1}\frac{ch_{\gamma_j}(i^*\sigma)}{ch_{\gamma_j}(\lambda_{-1}N_{\mathbb{C}})e}=\frac{1}{p}\bigg[-\frac{1}{2}e(7p-15)+\hat{h}\big(4-\frac{5}{6}(p^2-1)\big)\bigg].
\end{align}

\subsection{Computation of the index}
Using formula $\eqref{sum}$, we begin to compute the index:
\begin{align}
\begin{split}
\label{ind}
Ind^{ASD}(M,g)&=\frac{1}{2}(15\chi_{top}(M)+29\tau_{top}(M))\\
&\phantom{==}-\frac{15}{2}(1-p^{-1})\chi(\Sigma)-\frac{29}{6}\bigg(\frac{p^2-1}{p}\bigg)[\hat{\Sigma}]^2\\
&\phantom{==}-\bigg\langle\frac{1}{p}\bigg[-\frac{1}{2}e(7p-15)+\hat{h}(4-\frac{5}{6}(p^2-1))\bigg]\hat{A}(\Sigma)^2,[\Sigma]\bigg\rangle.
\end{split}
\end{align}

For a real oriented plane bundle $E$, whose complexification decomposes into complex line bundles as $E_{\mathbb{C}}=l\oplus \bar{l}$, we have that
\begin{align}
\begin{split}
\hat{A}(E)^2&=Td_{\mathbb{C}}(E_{ \mathbb{C}})=Td(l\oplus \bar{l})=Td(l)Td(\bar{l})\\
&=(1+\frac{1}{2}c_1(l)+\frac{1}{12}c_1(l)^2+\cdots)(1-\frac{1}{2}c_1(l)+\frac{1}{12}c_1(l)^2+\cdots)\\
&=1-\frac{1}{12}c_1(l)^2+\cdots.
\end{split}
\end{align}
In the fourth term on the right hand side of $\eqref{ind}$ we see that $\hat{A}^2(\Sigma)$ is only multiplied by terms containing Euler classes of the tangent and normal bundles of $\Sigma$.  Since this product is paired with the fundamental class of a surface, it is clear that only the first term in $\hat{A}^2(\Sigma)$ contributes to the index.  Therefore
\begin{align*}
Ind^{ASD}(M,g)&=\frac{1}{2}(15\chi_{top}(M)+29\tau_{top}(M))+\frac{1}{p}\bigg[\bigg(\frac{7}{2}p-\frac{15}{2}\bigg)+\bigg(\frac{15}{2}-\frac{15}{2}p\bigg)\bigg]\chi(\Sigma)\\
&\phantom{==}+\frac{1}{p}\bigg[-4+\frac{5}{6}(p^2-1)-\frac{29}{6}(p^2-1)\bigg][\hat{\Sigma}]^2\\
&=\frac{1}{2}(15\chi_{top}(M)+29\tau_{top}(M))-4\chi(\Sigma)-4p[\hat{\Sigma}]^2.
\end{align*}
Since $p[\hat{\Sigma}]^2=[\Sigma]^2$ we arrive at formula \eqref{Ind}:
\begin{align}
Ind^{ASD}(M,g)=\frac{1}{2}(15\chi_{top}(M)+29\tau_{top}(M))-4\chi(\Sigma)-4[\Sigma]^2.
\end{align}

It is clear from examining the sign changes in the above computations that the formula for the self-dual complex is
\begin{align}
Ind^{SD}(M,g)=\frac{1}{2}(15\chi_{top}(M)-29\tau_{top}(M))-4\chi(\Sigma)+4[\Sigma]^2,
\end{align}
which is formula \eqref{Ind2}.

Finally, when $\Sigma$ is a non-orientable surface the formulas \eqref{Ind} and
\eqref{Ind2} still hold.
This is proved by evaluating the pullbacks of the respective Euler classes to the orientable double cover, evaluating on that fundamental class and then dividing by~$2$.

\section{Proofs of Corollaries}
\label{coro}

We begin this section with a proposition bounding $\dim(H^0)$ for an 
orbifold-cone metric on $(M,\Sigma)$, which will be very useful in the following proofs.

\begin{proposition}
\label{bound}
Let $g$ be an orbifold-cone metric on $(M,\Sigma)$.  Then
\begin{align}
\dim(H^0)\leq 11,
\end{align}
with equality possible only if $\Sigma=S^2$.  Moreover
\begin{align}
\dim(H^0)\leq
\begin{cases}
7 \text{ when $\Sigma=T^2$}\\
5 \text{ when $\Sigma=j\#T^2$ for $j>1$}\\
8\text{ when $\Sigma=\RP^2$}\\
7 \text{ when $\Sigma=\RP^2\#\RP^2$}\\
5 \text{ when $\Sigma=j\#\RP^2$ for $j>2$}.
\end{cases}
\end{align}
\end{proposition}

\begin{proof}
We begin by proving bounds on the size of the conformal automorphism group.  The proof follows the idea of
Bagaev-Zhukova \cite{BagaevZhukova}, and we briefly recall their argument here:

We have the natural homomorphsim
\begin{align}
\phi:Conf(M,\Sigma)\rightarrow Conf(\Sigma),
\end{align}
where $Conf(M,\Sigma)$ and $Conf(\Sigma)$ are the conformal automorphism group of $(M,\Sigma)$ and $\Sigma$ respectively,
and an exact sequence
\begin{align}
1\rightarrow Ker(\phi)\rightarrow Conf(M,\Sigma) \rightarrow Im(\phi) \rightarrow 1.
\end{align}
This implies that $Conf(M,\Sigma)=Ker(\phi)\rtimes Im(\phi)$.  

Consider an element $f\in Ker(\phi)$ and its pushforward map $f_*:TM\rightarrow TM$.  This induces the maps $f_*|_{T\Sigma}=Id$ since $f|_{\Sigma}=Id$, and  $f_*|_{N\Sigma}\in \rm{O}(2)$.  So we get a homomorphism
\begin{align}
\alpha:Ker(\phi)\rightarrow\rm{O}(2)
\end{align}
by sending $f\mapsto f^*|_{N\Sigma}\in \rm{O}(2)$.  Therefore 
\begin{align}
\dim(Ker(\phi))\leq \dim (\rm{O}(2))+\dim (Ker(\alpha))\leq 5,
\end{align}
since $\dim (Ker(\alpha))$ is less than or equal to the dimension of the first prolongation of the Lie algebra of the $G$-structure group, which is $4$.  Therefore, $\dim(Conf(M,\Sigma))\leq\dim(Conf(\Sigma))+5$.

Now, $\dim(Im(\phi))\leq \dim(Conf(\Sigma))\leq 6$, with equality if and only if $\Sigma=S^2$.  We also know that
 \begin{align*}
\dim(Conf(j\#T^2))=
\begin{cases}
2 \text{ for $j=1$}\\
0 \text{ for $j\geq 2$}
\end{cases}
\text{ and }
\dim(Conf(j\#\RP^2))=
\begin{cases}
3 \text{ for $j=1$}\\
2 \text{ for $j=2$}\\
0 \text{ for $j\geq 3$}.
\end{cases}
\end{align*}
We complete the proof by showing that 
\begin{align}
\dim(H^0)=\dim(Conf(M,\Sigma)\leq\dim(Conf(\Sigma))+5.
\end{align}
In the smooth case, a 
conformal Killing field is the derivative of a $1$-parameter family of conformal transformations.  
Thus, the space of conformal Killing fields is identified with the Lie algebra of the conformal automorphism group.
For the orbifold case take a neighborhood $U_q$ around $q\in \Sigma$ and lift it to a neighborhood $\tilde{U}_q$ around the origin in $\mathbb{R}^4$, on which the metric $g$ pulls back to $\tilde{g}$, a $\Gamma$-invariant metric.  Any conformal Killing field on $U_q$ lifts to a $\Gamma$-invariant conformal Killing field on $\tilde{U}_q$ since $\tilde{g}$ is $\Gamma$-invariant.  The local $1$-parameter families of diffeomorphisms on $\tilde{U}_q$ coming from the lifts of these conformal Killing fields must also be $\Gamma$-invariant since the flow is locally defined.  From the uniqueness of the flow 
on each $\tilde{U}_q$, these patch together to give a globally defined $1$-parameter group of conformal transformations 
on the orbifold.  Therefore, in the orbifold case we also have, $\dim(H^0)=\dim(Conf(M,\Sigma))$.

\end{proof}

The fact that a self-dual positive scalar curvature Einstein orbifold-cone metric is unobstructed is crucial to the proof of 
Corollary \ref{c1}, so we now prove Lemma \ref{lemma}:

\begin{proof}[Proof of Lemma \ref{lemma}]
Let $g$ be an anti-self-dual Einstein orbifold-cone metric with positive scalar curvature on $M$ with singular set $\Sigma$. 
Let $Z\in Ker(\mathcal{D}^{+^*})$.  Then $\mathcal{D}^+\mathcal{D}^{+^*}Z=0$, which implies
\begin{align}
\label{itoh}
\mathcal{D}^+\mathcal{D}^{+^*}Z=\frac{1}{24}(3\nabla^*\nabla+2R)(2\nabla^*\nabla+R)Z=0,
\end{align}
using the Weitzenb\"ock formula of Itoh in the case that $g$ is Einstein \cite{Itoh}, where $R$ is the scalar curvature.  Also, recall that $\nabla^*\nabla=\Delta$, 
the rough laplacian.

Cut out $N_{\epsilon}$, an $\epsilon$-tubular neighborhood of $\Sigma$ and denote the outer unit
normal vector and induced volume form on $\partial (M\setminus N_{\epsilon})$ by $n$ and $d\sigma$ respectively.
Using Itoh's Weitzenb\"ock formula \eqref{itoh} and integrating by parts, we see that
\begin{align}
\begin{split}
\label{integral}
\int \limits_{M\setminus N_{\epsilon}}\langle \mathcal{D}^+\mathcal{D}^{+^*}Z,Z\rangle &dV
=\int \limits_{M\setminus N_{\epsilon}}\Big[\frac{1}{4}|\Delta Z|^2+\frac{7}{24}R|\nabla Z|^2
+\frac{1}{12}R^2|Z|^2\Big]dV\\
&-\int \limits_{\partial (M\setminus N_{\epsilon})}\Big[ \frac{7}{48}R\nabla_n |Z|^2
+\frac{1}{4} \langle \Delta Z,  \nabla_nZ\rangle -\frac{1}{4} \langle \nabla_n\Delta Z,  Z\rangle \Big] d\sigma.
\end{split}
\end{align}

Since $Z$ is a smooth section in the orbifold sense, $Z$ and its derivatives are bounded.  Therefore, by dominated convergence, the solid integrals limit to the corresponding solid integrals on $M$ as $\epsilon \rightarrow 0$.  
For $\epsilon$ sufficiently small, $\partial (M\setminus N_{\epsilon})$ is a smooth submanifold, and we have the estimate
\begin{align}
Area(\partial (M\setminus N_{\epsilon}))<C\epsilon,
\end{align}
for some constant $C$.  Consequently the boundary integral limits to $0$ as $\epsilon \rightarrow 0$, and we have that
\begin{align}
\int\limits_{M}\langle \mathcal{D}^+\mathcal{D}^{+^*}Z,Z\rangle dV=\int\limits_{M}\Big[\frac{1}{4}|\Delta Z|^2+\frac{7}{24}R|\nabla Z|^2
+\frac{1}{12}R^2|Z|^2\Big]dV\geq 0
\end{align}
with equality if and only if $Z=0$, since $R>0$.  Therefore $H^2=\{0\}$.
The proof for self-dual metrics is analogous.

\end{proof}

Finally, we prove the corollaries from the Introduction:

\begin{proof}[Proof of Corollary \ref{c1}]
For $k=3$ this is the standard metric on $S^4$, which is rigid.
Hitchin's metrics, $\{g_k\}_{k\geq 4}$, all have singular set $\Sigma = \RP^2$ 
with  self-intersection $-2$. Theorem \ref{mainthm} for self-dual 
metrics implies that 
\begin{align}
Ind^{SD}(S^4,g_k) = 15 - 4 \cdot 1 - 4 \cdot 2 = 3.
\end{align}
Since each $g_k$ is an unobstructed self-dual orbifold-cone metric with 
$\dim(H^0) = 3$, we have
\begin{align}
3 - \dim(H^1) = 3, 
\end{align}
which implies that $\dim(H^1) = 0$. Consequently, using the remarks 
about the Kuranishi map from the Introduction, these metrics are rigid. 
\end{proof}

\begin{proof}[Proof of Corollary \ref{orientable}]
Let $g$ be an unobstructed self-dual orbifold-cone metric on $(S^4,\Sigma^j)$, where $\Sigma^j$ is diffeomorphic 
to an orientable surface of genus $j\geq1$.  We know that $\chi(\Sigma^j)=2-2j$ and $[\Sigma^j]^2=0$.  Therefore 
\begin{align}
\begin{split}
Ind^{SD}(S^4,g)&=\frac{1}{2}(15\cdot 2-29\cdot 0)-4(2-2j)+4\cdot 0\\
&=7+8j.
\end{split}
\end{align}
Since $g$ is unobstructed, using Proposition \ref{bound} we have the inequality
\begin{align}
\label{ine}
7-\dim(H^1)\geq \dim(H^0)-\dim(H^1)=7+8j.
\end{align}
However, since $j\geq 1$ this inequality cannot hold, which 
proves the second part of the corollary.
\end{proof}

\begin{proof}[Proof of Corollary \ref{c2}]
Let $g$ be an unobstructed self-dual orbifold-cone metric on $S^4$ with singular set $\Sigma^j$, 
a smoothly embedded surface diffeomorphic to $j\#\RP^2$.  
We know that $\chi(j\#\RP^2)=2-j$.  Therefore
\begin{align}
\begin{split}
Ind^{SD}(S^4,g)&=\frac{1}{2}(15\cdot 2-29\cdot 0)-4(2-j)+4[\Sigma^j]^2\\
&=7+4j+4[\Sigma^j]^2.
\end{split}
\end{align}
Since $g$ is unobstructed we have the inequality
\begin{align}
7+4j+4[\Sigma^j]^2 \leq \dim(H^0).
\end{align}
Using the bounds on $\dim(H^0)$ in Proposition \ref{bound} we see that
\begin{align}
[\Sigma^j]^2\leq
\begin{cases}
-\frac{3}{4} &\text{ when $j=1$}\\
-2 &\text{ when $j=2$}\\
-\frac{1}{2}-j &\text{ when $j\geq 3$}.
\end{cases}
\end{align}
Combining these restrictions with the list of possible values of the 
self-intersection numbers (proved by Massey) \eqref{values}, completes the proof.
\end{proof}

\begin{proof}[Proof of Corollary \ref{c4}]

Let $g$ be an unobstructed LeBrun self-dual orbifold-cone metric on $n \# \CP^2$, with 
singular set $\Sigma = S^2$ and cone angle $2 \pi /p$. 
Clearly, for $n \geq 3$, the moduli space of such metrics 
is of dimension $3n - 6$, which is obtained by counting 
the moduli space of monopole points modulo the action 
of the group of hyperbolic isometries. 
We know that
\begin{align}
\chi(\Sigma)=(n+2) \text{ and } [\Sigma]^2=n.
\end{align}
Therefore
\begin{align}
Ind^{SD}(n\# \mathbb{CP}^2,g)&=\frac{1}{2}(15(n+2)-29\cdot n)-4\cdot 2+4\cdot n =-3n+7.
\end{align}
Since $g$ is unobstructed and $\dim(H^0)=1$ we have that
\begin{align}
\dim(H^1)=3n-6.
\end{align}
Since the dimension of the moduli space is greater than or equal 
to $3n - 6$, the action of $H^0$ on $H^1$ must be trivial. 
Therefore, the dimension of the moduli space is 
exactly $3n - 6$, so any sufficiently close
self-dual deformation of $g$ must be $S^1$-equivariant.
\end{proof}

\begin{proof}[Proof of Corollary \ref{c5}]

We can see that $H^2$ consists of parallel sections of $S^2_0(\Lambda^2_+)$ using the argument in the proof of 
Lemma \ref{lemma} with $R=0$, so $H^2 = \{0\}$ by assumption.
Similarly, since $Ric=0$, the standard Bochner argument works in the 
orbifold-cone setting to show that $H^0$ consists of parallel vector fields, and therefore
 $H^0 = \{0\}$, also by assumption.
Consequently, using the facts about the Kuranishi map from the 
Introduction, the moduli space is smooth near $g$, and the 
dimension is computed by Theorem \ref{mainthm}, yielding \eqref{rfdim}.

\end{proof}

\bibliography{ASD_Edge_Cone_references}

\end{document}